\documentclass[letterpaper,reqno,11pt]{amsart}

\usepackage{amsmath,amssymb}
\usepackage{accents}
\usepackage{graphicx}
\usepackage{verbatim}
\usepackage{enumitem}
\usepackage{tikz}
\usepackage{caption}
\usepackage{subcaption}
\usepackage{tabu}
\usepackage{longtable}
\usepackage{array}
\usepackage{hyperref}
\usepackage{multirow}
\usepackage{tablefootnote}
\usepackage{appendix}

\newlength{\dhatheight}

\theoremstyle{definition}

\newtheorem{theorem}{Theorem}[section]

\newtheorem{proposition}[theorem]{Proposition}

\newtheorem{corollary}[theorem]{Corollary}
\newtheorem{definition}[theorem]{Definition}

\newtheorem{method}[theorem]{Method}
\newtheorem{algorithm}[theorem]{Algorithm}

\newcommand{\snappy}{{\sf SnapPy}}
\newcommand{\Heegaard}{{\sf Heegaard}}

\newcommand{\Z}{{\mathbb Z}}

\newcommand{\specialcell}[2][c]{%
	\begin{tabular}[#1]{@{}c@{}}#2\end{tabular}}
\newcommand{\brackets}[1]{  \left\{ {#1} \right\}  }

\DeclareMathOperator{\rank}{rank}
\newcommand{\Mod}[1]{\mathrm{mod}\ #1}
\newcommand{\abs}[1]{\left\lvert #1\right\rvert}

\begin{document}

\title[The tunnel number of all 11 and 12 crossing alternating knots]
{The tunnel number of all 11 and 12 crossing alternating knots}

\author[F.~Castellano-Mac\'ias]{Felipe Castellano-Mac\'ias}

\author[N.~Owad]{Nicholas Owad}

\address{Northeastern University\\
Boston, MA 02115, USA\\
and\\
Topology and Geometry of Manifolds Unit\\
Okinawa Institute of Science and Technology Graduate University\\
Okinawa, Japan 904-0495}
\email{castellanomacias.f@northeastern.edu}

\address{Topology and Geometry of Manifolds Unit\\
Okinawa Institute of Science and Technology Graduate University\\
Okinawa, Japan 904-0495\\
and\\
Department of Mathematics\\
Hood College\\
Frederick, MD 21701, USA}
\email{owad@hood.edu}

\thanks{2016 {\em Mathematics Subject Classification}. 57M25, 57M27}

\begin{abstract}
Using exhaustive techniques and results from Lackenby and many others, we compute the tunnel number of all 1655 alternating 11 and 12 crossing knots and of 881 non-alternating 11 and 12 crossing knots. We also find all 5525 Montesinos knots with 14 crossings or fewer.
\end{abstract}

\maketitle

\section{Introduction}

Tunnel number is a knot invariant, first defined by Clark in 1980, \cite{Clark}. The tunnel number $t(K)$ can be realized as one less than the Heegaard genus of $S^3 \setminus N(K)$, or as the minimum number of properly embedded disjoint arcs $\alpha_i$ required to make $S^3 \setminus N(K \cup \{\alpha_i\})$ a handlebody. Since it was defined, it has become a classical invariant with connections to the hyperbolic volume of knots \cite{KR}, bridge number, and many others. It has interesting and unexpected properties under connected sum \cite{MSY1} and is a common tool used to investigate characteristics of knots and links. Because of this, it is useful to have explicit values for tunnel number to test conjectures against. For an overview of tunnel number, see Morimoto \cite{Morimoto}.

A paper by Morimoto, Sakuma, and Yokota \cite{MSY2} computed the tunnel number of all knots with 10 or fewer crossings, of which there are 250. This paper aims to extend this list of known values of tunnel numbers. The main source of values of tunnel numbers here is Lackenby's paper \cite{Lackenby}, where he proves a conjecture of Sakuma and classifies all tunnel number one alternating knots. Briefly, the main theorem of Lackenby is as follows: $K$ is an alternating, tunnel number one knot if and only if $K$ is either a 2-bridge knot or a 3-bridge Montesinos knot with a {\em clasp}. A clasp is a rational tangle with corresponding rational number $\pm \frac{1}{2}$. 

We enumerate all possible examples of these knots with 11 and 12 crossings and use the program \snappy \ \cite{snappy}, by Culler, Dunfield, Goerner, and Weeks, to identify them. We also use the data on bridge number of these knots, supplied by the online database Knotinfo \cite{knotinfo}, and are able to identify the tunnel number of every alternating 11 and 12 crossing knot, of which there are 1655 knots. For exactly two alternating knots, this first method does not work, but Moriah and Lustig's result in \cite{LM} provides the tunnel number. These and other methods also give exact values and bounds for many non-alternating knots. There are 1073 non-alternating 11 and 12 crossing knots, and we have calculated the tunnel numbers of 881 of them. The remaining 192 knots have tunnel number one or two.

\begin{theorem}\label{thm:main}
The tunnel number of all 1655 alternating 11 or 12 crossing knots has been calculated. The tunnel number of 881 non-alternating 11 or 12 crossing knots has also been calculated.
\end{theorem}

A list of these tunnel numbers can be found in Appendix \ref{appendix}. In the process of this work, we have also enumerated all Montesinos knots with 14 crossings or fewer. The full list of all these results and the associated code can be found online for download at
\begin{center}
	\url{https://github.com/fcastellanomacias/tunnel}.
\end{center}  

We will define the relevant terms in the next section and in Section \ref{sec:bounds} we give well-known relations between tunnel number and other invariants. Then in Section \ref{sec:main} we will discuss the algorithms used and finally, in Section \ref{sec:results}, we prove our main result and provide summary tables which help explain the proof.

\subsection*{Acknowledgments} The first author would like to thank the Okinawa Institute of Science and Technology for their hospitality throughout their internship there. We also thank Nathan Dunfield for useful suggestions regarding symmetries of knots and Method \ref{method}, and Ken Baker for helpful comments.

\section{Definitions}

We assume the reader is familiar with the basics of knot theory, see \cite{Rolfsen} for background. Throughout this paper, we will assume all knots have a single component. The two main families of knots we will need to consider are rational knots and Montesinos knots; both are built from rational tangles, see \cite[Chapter 12]{BZ} for more details.

Throughout this paper, we adopt the following convention for continued fractions:
	\[ \left[ a_1, a_2, \ldots, a_m \right] := a_{1} + \frac{1}{a_{2} + \dots + \frac{1}{a_{m-1}+\frac{1}{a_{m}}}}. \]

\begin{definition}\label{def:rational}
A {\em rational tangle} for $(\alpha,\beta)$ as illustrated in Figure \ref{fig:rationaltangle} is defined by the continued fraction $\frac{\beta}{\alpha} = [ a_1,-a_2,a_3, \ldots, \pm a_m]$, $ a_j = a'_j + a '' _j$, together with the conditions that $\alpha$ and $\beta$ are relatively prime and $\alpha>0$. \end{definition}

\begin{figure}[h]
\begin{center}
\begin{tikzpicture}[scale =.5]

\def\w {.3}

\foreach \c in {-2,-1,0,1,2}{
\draw [ line width = \w mm, line cap=round] (\c+.5,1) to (\c-.5,-1);
\draw [ line width = 3*\w mm, line cap=round, white] (\c-.25,.5) to (\c+.5,-1);
\draw [ line width = \w mm, line cap=round] (\c-.5,1) to (\c+.5,-1);
}

\foreach \c in {3.25,3.5,-5.75,-5.5,-3-5/8,5+3/8}{
\draw [ line width = \w mm, line cap=round] (-1,\c+.25) to (1,\c);
\draw [ line width = 2.5*\w mm, line cap=round, white] (-1/2,\c+.25/4) to (1,\c+.25);
\draw [ line width = \w mm, line cap=round] (-1,\c) to (1,\c+.25);
}

\foreach \c in {-6.75,-6.5,-4.875,-4.625,-4.375,4.25,4.5}{
\draw [ line width = \w mm, line cap=round] (\c,-1) to (\c+.25,1);
\draw [ line width = 2.5*\w mm, line cap=round, white] (\c+.25,-1) to (\c+.25/4,1-.5/1);
\draw [ line width = \w mm, line cap=round] (\c+.25,-1) to (\c,1);

}

\draw [ line width = \w mm, line cap=round] (2.5,1) to (3,1) to (3,3) to (1,3) to (1,3.25);
\draw [ line width = \w mm, line cap=round] (2.5,-1) to (3,-1) to (3,-3) to (1,-3) to (1,-3-3/8);
\draw [ line width = \w mm, line cap=round] (-2.5,1) to (-3,1) to (-3,3) to (-1,3) to (-1,3.25);
\draw [ line width = \w mm, line cap=round] (-2.5,-1) to (-3,-1) to (-3,-3) to (-1,-3) to (-1,-3-3/8);

\draw [ line width = \w mm, line cap=round]  (-1,-3-5/8) to (-1, -4) to (-4,-4) to (-4,-1) to (-4-1/8,-1);
\draw [ line width = \w mm, line cap=round]  (1,-3-5/8) to (1, -4) to (4,-4) to (4,-1) to (4+2/8,-1);
\draw [ line width = \w mm, line cap=round]  (-1,4-2/8) to (-1,4) to (-4,4) to (-4,1) to (-4-1/8,1);
\draw [ line width = \w mm, line cap=round]  (1,4-2/8) to (1, 4) to (4,4) to (4,1) to (4.25,1);

\draw [ line width = \w mm, line cap=round] (-4-7/8,-1) to (-5,-1) to (-5,-5) to (-1,-5) to (-1,-5.25);
\draw [ line width = \w mm, line cap=round] (4+6/8,-1) to (5,-1) to (5,-5) to (1,-5) to (1,-5.25);
\draw [ line width = \w mm, line cap=round] (-4-7/8,1) to (-5,1) to (-5,5) to (-1,5) to (-1,5+3/8);
\draw [ line width = \w mm, line cap=round] (4+6/8,1) to (5,1) to (5,5) to (1,5) to (1,5+3/8);

\draw [ line width = \w mm, line cap=round] (-1,-5.75) to (-1,-6) to (-6,-6) to (-6,-1) to (-6.25,-1);
\draw [ line width = \w mm, line cap=round] (-1,5+5/8) to (-1,6) to (-6,6) to (-6,1) to (-6.25,1);
\draw [ line width = \w mm, line cap=round] (1,-5.75) to (1,-6) to (6,-6) to (6,-1) to (7,-1) to (7,-7) to (2,-7) to (2,-8);
\draw [ line width = \w mm, line cap=round] (1,5+5/8) to (1,6) to (6,6) to (6,1) to (7,1) to (7,7) to (2,7)  to (2,8);

\draw [ line width = \w mm, line cap=round] (-6.75,-1) to (-7, -1) to (-7,-7) to (-2,-7) to (-2,-8);
\draw [ line width = \w mm, line cap=round] (-6.75,1) to (-7, 1) to (-7,7) to (-2,7) to (-2,8);
\draw [ dashed, thick] (3,1) to [out=0,in=0] (3,-1);
\draw [fill] (3,1) circle [radius=0.1];
\draw [fill] (3,-1) circle [radius=0.1];

\node [] at (0,1.75) {$a_5$};

\node [] at (1.75,3.5) {$a'_4$};
\node [] at (1.75,-3.5) {$a''_4$};

\node [] at (-4.5,1.75) {$a'_3$};
\node [] at (4.5,1.75) {$a''_3$};

\node [] at (1.75,5.5) {$a'_2$};
\node [] at (1.75,-5.5) {$a''_2$};

\node [] at (-6.5,1.75) {$a'_1$};
\node [] at (6.5,1.75) {$a''_1$};

\end{tikzpicture}\caption{A rational tangle for $(181,297) $, where $\frac{297}{181}=[2,-3,5,-3,5]$. Note that $a_j'$ and $a_j''$ are the number of half-twists. The dashed line represents an unknotting tunnel.}
\label{fig:rationaltangle}
\end{center}
\end{figure}

Schubert defined the bridge number of a knot in 1954, \cite{Schubert}. 

\begin{definition}\label{def:bridgenum}
The {\em bridge number} of a diagram $D$ is the minimum number of local maxima of $D$. The {\em bridge number} $b(K)$ of a knot $K$ is the minimum number of the bridge numbers over all diagrams $D$ of $K$.
\end{definition}

Given a knot $K$, $b(K) = 1 $ if and only if $K$ is the unknot. So, under this invariant, the first class of nontrivial knots are 2-bridge knots and Schubert completely classified these knots. Another name for 2-bridge knots is {\em rational knots}, named so because they are a composed of a single rational tangle with a numerator (or denominator) closure.

\begin{figure}[h]
\begin{center}
\begin{tikzpicture}[scale =.6]

\def\w {.3}
\begin{scope}[rotate = -90, scale = 1]

\draw[line width = \w mm] (-1.2*.707,1.2*.707) to [out=135,in=180+135] (-1.6*.707,1.6*.707) to [out=135,in=180] (0,2) to  [out=0,in=45] (1.6*.707,1.6*.707)to  [out=225,in=45] (1.2*.707,1.2*.707);

\draw[line width = \w mm] (-1.2*.707,-1.2*.707) to [out=225,in=180+225] (-1.6*.707,-1.6*.707) to [out=225,in=180] (0,-2) to  [out=0,in=-45] (1.6*.707,-1.6*.707)to  [out=-225,in=-45] (1.2*.707,-1.2*.707);

\draw [line width = \w mm] (0,0) circle [radius=1.2];

\node [] at (0,0) {\small $T$};

\end{scope}


\begin{scope}[xshift = 6cm, scale = 1]

\draw[line width = \w mm] (-.5,1.6) to [out=225,in=90] (-1.1,0) to [out=270, in=135] (-.5,-1.6);
\draw[line width = \w mm] (.5,1.6) to [out=-45,in=90] (1.1,0) to [out=270, in=45] (.5,-1.6);

\draw[line width = \w mm] (-.25,-1.6) to [out=135,in=180] (0,.25);
\draw[line width = 3*\w mm, white] (-.25,1.6) to [out=135,in=180] (0,-.25);
\draw[line width = \w mm] (-.25,1.6) to [out=135,in=180] (0,-.25);

\draw[line width = \w mm] (0,-.25) to [out=0, in=45] (.25,1.6);
\draw[line width = 3*\w mm, white] (0,.25) to [out=0, in=45] (.25,-1.6);
\draw[line width = \w mm] (0,.25) to [out=0, in=45] (.25,-1.6);

\draw [fill, white] (0,2) circle [radius=1.2];
\draw [fill, white] (0,-2) circle [radius=1.2];

\draw [line width = \w mm] (0,2) circle [radius=1.2];
\draw [line width = \w mm] (0,-2) circle [radius=1.2];

\node [] at (0,2) {\small $T_1$};
\node [] at (0,-2) {\small $T_2$};

\end{scope}


\begin{scope}[xshift = 10cm, yshift = -6.4cm, scale = .8]

\draw (9.5,3) rectangle +(2,2);

\draw (9.5,9) rectangle +(2,2);
\draw (9.5,12) rectangle +(2,2);

\draw [domain=0:.4] plot ({.5*cos(pi* \x r )+1.5}, \x+6);

\draw [domain=0:1.4] plot ({- .5*cos(pi* \x r )+1.5}, \x+6);

\draw [domain=.6:2.4] plot ({.5*cos(pi* \x r )+1.5}, \x+6);

\draw [domain=.6:2.4] plot ({.5*cos(pi* \x r )+1.5}, \x+7);

\draw [domain=.6:2] plot ({.5*cos(pi* \x r )+1.5}, \x+8);
\draw [domain=.6:1] plot ({.5*cos(pi* \x r )+1.5}, \x+9);

\draw [rounded corners] (1,6) -- (1,1) -- (11,1)--(11,3);
\draw [rounded corners] (2,6) -- (2,2) -- (10,2)--(10,3);

\draw [rounded corners] (1,10) -- (1,16) -- (11,16)--(11,14);
\draw [rounded corners] (2,10) -- (2,15) -- (10,15)--(10,14);

\foreach \a in {5,8,11}{
	\draw (10,\a) -- (10,\a+1);
	\draw (11,\a) -- (11,\a+1);
	}

\draw [fill] (10.5,6.7) circle [radius=0.05];	
\draw [fill] (10.5,7) circle [radius=0.05];	
\draw [fill] (10.5,7.3) circle [radius=0.05];	

\node at (0.5,8) {\small$e$};

\node at (10.5, 13) {\small $\beta_1/\alpha_1$};
\node at (10.5, 10) {\small$\beta_2/\alpha_2$};
\node at (10.5, 4) {\small$\beta_n/\alpha_n$};

\end{scope}

\end{tikzpicture}\caption{A rational knot, with a numerator closure, on the left. In the middle, a clasp Montesinos link with $e=0$, and on the right, a Montesinos link.}
\label{fig:rationalKnot}
\end{center}
\end{figure}

{\em Montesinos knots} are a generalization of rational knots, first introduced by Montesinos in 1973, \cite{Montesinos}.

\begin{definition}\label{def:Montesinos}
The {\em Montesinos link} $M(e;\beta_1 / \alpha_1, \beta_2/\alpha_2, \ldots, \beta_r/\alpha_r)$ is a link admitting a diagram like that of Figure \ref{fig:rationalKnot}, where each box represents a rational tangle. Also, $r\geq 3$ and $\beta_i/\alpha_i$ is not an integer, otherwise $M$ would have a simpler diagram. If $e$ is left out of the notation, we assume $e=0$. 
\end{definition}

For convenience, we will make a more specific class of Montesinos links that will be useful later. See Figure \ref{fig:rationalKnot}.

\begin{definition}\label{def:claspMont}
A {\em clasp Montesinos link} is a Montesinos link where exactly one of the $\beta_i/\alpha_i = \pm 1/2$. 
\end{definition}

By \cite[Theorem 12.29]{BZ}, Montesinos links are equivalent up to cyclic permutations of the fractions and up to the value of each fraction mod 1. This shows that our definition of clasp Montesinos link is well-defined.

And finally, we define the main invariant that we are considering in this paper.

\begin{definition}\label{def:tunnel}
Given a knot $K$, a {\em tunnel} is a properly embedded arc in $S^3\setminus N(K)$, where $N(\cdot)$ is an open regular neighborhood. The {\em tunnel number} $t(K)$ of $K$ is the minimum number of disjoint tunnels $\alpha_i$ required to make $S^3 \setminus N(K \cup \{\alpha_i\})$ a handlebody.
\end{definition}

The dashed line in Figure \ref{fig:rationaltangle} is a tunnel. Contract the tunnel to a point, creating a graph which, through ambient isotopy, allows all the crossings in the center of the tangle to be undone.  Then the next innermost crossings can be undone, and so on, until we have a graph that is shaped like an ``X''. For ease, we will refer to this process as {\em collapsing a tangle to a point.}

\section{Bounds on tunnel number}\label{sec:bounds}

We begin this section with a well-known proposition which relates the rank of the knot group, tunnel number, and bridge number.

\begin{proposition}[\cite{LM}]\label{thm:tunnelbridge}
	For any knot $ K $, we have that \[ \rank(\pi_{1}(S^{3} \setminus K)) -1 \leq t(K) \leq b(K) - 1. \]
\end{proposition}

In addition, the following theorem shows that a Montesinos knot with $ r $ rational tangles has bridge number equal to $ r $.

\begin{theorem}[{\cite[Theorem 1.1]{BoileauZieschang}}] 
	Let $ K $ be the Montesinos knot $ M(e; \beta_1/\alpha_1, \dots , \beta_r/\alpha_r) $, where $ \alpha_{i} \neq 1 $ for all $ i $. Then $ b(K) = r $.
\end{theorem}

Next, we generalize a lemma from Lackenby \cite{Lackenby} which gives an upper bound for the tunnel number of clasp Montesinos knots.

\begin{proposition}\label{thm:clasp}
	Let $ K $ be the clasp Montesinos knot $ M(e; \beta_1/\alpha_1, \dots , \beta_r/\alpha_r) $. Then $ t(K) \leq r-2 $.
\end{proposition}

\begin{proof} We follow an argument similar to \cite{Lackenby}. Without loss of generality, we may assume that $ K $ is of the form $ M(0; \beta_1/\alpha_1, \dots, \beta_{r-2}/\alpha_{r-2}, 1/2, \beta_r/\alpha_r) $. As shown in Figure \ref{fig:rationaltangle}, we can place a tunnel at the center of each rational tangle to {\em collapse the tangle to a point}. Thus, after placing a tunnel in each rational tangle $ \beta_i/\alpha_i $, for $ 1 \leq i \leq r-2 $, we obtain the diagram in the middle of Figure \ref{fig:proof}. In this diagram, we can now slide down the outermost arc from the top to the bottom vertex via ambient isotopy without altering the exterior, obtaining the diagram on the right of Figure \ref{fig:proof}. As shown by Lackenby in \cite{Lackenby}, we obtain that $ t(K) \leq r-2 $. Note that Lackenby's proof also applies when $ K $ is non-alternating.
\begin{figure}[h]
	\begin{center}
		\begin{tikzpicture}[scale =0.6]
		
		\def\w {.3}
		
\begin{scope}[xshift = -6.5cm, scale = 1]

\draw[line width = \w mm] (0,5) circle [radius=1];
\draw[line width = \w mm] (0,2) circle [radius=1];
\draw[line width = \w mm] (0,-2) circle [radius=1];
\draw[line width = \w mm] (0,-6) circle [radius=1];

\draw[line width = \w mm] (-.707,4+.293) to [out=180+45, in=90+45] (-.707,3-.293);
\draw[line width = \w mm] (.707,4+.293) to [out=0-45, in=45] (.707,3-.293);
\draw[line width = \w mm] (-.707,1+.293) to [out=180+45, in=90+45] (-.707,-.293);
\draw[line width = \w mm] (.707,1+.293) to [out=0-45, in=45] (.707,-.293);
\draw[line width = \w mm] (-.707,+.293) to [out=180+45, in=90+45] (-.707,-1-.293);
\draw[line width = \w mm] (.707,+.293) to [out=0-45, in=45] (.707,-1-.293);
\draw[line width = \w mm] (-1,5) to [out=180+65, in=180-65] (-1,-6);
\draw[line width = \w mm] (1,5) to [out=270+25, in=90-25] (1,-6);

\draw[fill, white] (-1.5,0-0.4) rectangle (1.5,0.4); 

\draw[fill] (0,0) circle [radius=0.04];
\draw[fill] (0,0.2) circle [radius=0.04];
\draw[fill] (0,-0.2) circle [radius=0.04];

\draw[line width = \w mm] (-.25,-4-1.6) to [out=135,in=180] (0,-4+.25);
\draw[line width = 4*\w mm, white] (-.25,1.6-4) to [out=135,in=180] (0,-.25-4);
\draw[line width = \w mm] (-.25,1.6-4) to [out=135,in=180] (0,-.25-4);

\draw[line width = \w mm] (0,-.25-4) to [out=0, in=45] (.25,1.6-4);
\draw[line width = 4*\w mm, white] (0,.25-4) to [out=0, in=45] (.25,-1.6-4);
\draw[line width = \w mm] (0,.25-4) to [out=0, in=45] (.25,-1.6-4);

\draw[fill, gray] (0,-2) circle [radius=1];
\draw[fill, gray] (0,-6) circle [radius=1];
\draw[line width = \w mm] (0,-2) circle [radius=1];
\draw[line width = \w mm] (0,-6) circle [radius=1];
\draw[fill, gray] (0,5) circle [radius=1];
\draw[fill, gray] (0,2) circle [radius=1];
\draw[line width = \w mm] (0,5) circle [radius=1];
\draw[line width = \w mm] (0,2) circle [radius=1];

\draw[fill] (-0.5,5) circle [radius=0.075];
\draw[fill] (0.5,5) circle [radius=0.075];
\draw[dashed, thick] (-0.5,5) to (0.5,5);

\draw[fill] (-0.5,2) circle [radius=0.075];
\draw[fill] (0.5,2) circle [radius=0.075];
\draw[dashed, thick] (-0.5,2) to (0.5,2);

\draw[fill] (-0.5,-2) circle [radius=0.075];
\draw[fill] (0.5,-2) circle [radius=0.075];
\draw[dashed, thick] (-0.5,-2) to (0.5,-2);

\draw[line width = \w mm] (-0.5,5) to [out=90, in=-45] (-0.5-0.15,5+0.25);
\draw[line width = \w mm] (-0.5,5) to [out=270, in=45] (-0.5-0.15,5-0.25);

\draw[line width = \w mm] (0.5,5) to [out=90, in=180+45] (+0.5+0.15,5+0.25);
\draw[line width = \w mm] (0.5,5) to [out=270, in=90+45] (+0.5+0.15,5-0.25);

\draw[line width = \w mm] (-0.5,2) to [out=90, in=-45] (-0.5-0.15,2+0.25);
\draw[line width = \w mm] (-0.5,2) to [out=270, in=45] (-0.5-0.15,2-0.25);

\draw[line width = \w mm] (0.5,2) to [out=90, in=180+45] (+0.5+0.15,2+0.25);
\draw[line width = \w mm] (0.5,2) to [out=270, in=90+45] (+0.5+0.15,2-0.25);

\draw[line width = \w mm] (-0.5,-2) to [out=90, in=-45] (-0.5-0.15,-2+0.25);
\draw[line width = \w mm] (-0.5,-2) to [out=270, in=45] (-0.5-0.15,-2-0.25);

\draw[line width = \w mm] (0.5,-2) to [out=90, in=180+45] (+0.5+0.15,-2+0.25);
\draw[line width = \w mm] (0.5,-2) to [out=270, in=90+45] (+0.5+0.15,-2-0.25);

\end{scope}

\begin{scope}
\draw[line width = \w mm] (0,-6) circle [radius=1];

\draw[line width = \w mm] (0,5) to [out=180+45, in=90+45] (0,2);
\draw[line width = \w mm] (0,5) to [out=0-45, in=45] (0,2);
\draw[line width = \w mm] (0,2) to [out=180+45, in=90] (-0.6,0.4);
\draw[line width = \w mm] (0,2) to [out=0-45, in=90] (0.6,0.4);
\draw[line width = \w mm] (-0.6,-0.4) to [out=270, in=90+45] (0,-2);
\draw[line width = \w mm] (0.6,-0.4) to [out=270, in=45] (0,-2);
\draw[line width = \w mm] (0,5) to [out=180+35, in=180-65] (-1,-6);
\draw[line width = \w mm] (0,5) to [out=-35, in=65] (1,-6);

\draw[fill] (0,0) circle [radius=0.04];
\draw[fill] (0,0.2) circle [radius=0.04];
\draw[fill] (0,-0.2) circle [radius=0.04];

\draw[line width = \w mm] (-.25,-4-1.6) to [out=135,in=180] (0,-4+.25);
\draw[line width = 4*\w mm, white] (0,-2) to [out=180+45,in=180] (0,-.25-4);
\draw[line width = \w mm] (0,-2) to [out=180+45,in=180] (0,-.25-4);

\draw[line width = \w mm] (0,-.25-4) to [out=0, in=-45] (0,-2);
\draw[line width = 4*\w mm, white] (0,.25-4) to [out=0, in=45] (.25,-1.6-4);
\draw[line width = \w mm] (0,.25-4) to [out=0, in=45] (.25,-1.6-4);

\draw[fill, gray] (0,-6) circle [radius=1];
\draw[line width = \w mm] (0,-6) circle [radius=1];

\draw[fill] (0,5) circle [radius=0.15];
\draw[fill] (0,2) circle [radius=0.15];
\draw[fill] (0,-2) circle [radius=0.15];

\end{scope}

\begin{scope}[xshift = 6.5cm, scale = 1]
\draw[line width = \w mm] (0,-6) circle [radius=1];

\draw[line width = \w mm] (0,5) to [out=180+45, in=90+45] (0,2);
\draw[line width = \w mm] (0,5) to [out=0-45, in=45] (0,2);
\draw[line width = \w mm] (0,2) to [out=180+45, in=90] (-0.6,0.4);
\draw[line width = \w mm] (0,2) to [out=0-45, in=90] (0.6,0.4);
\draw[line width = \w mm] (-0.6,-0.4) to [out=270, in=90+45] (0,-2);
\draw[line width = \w mm] (0.6,-0.4) to [out=270, in=45] (0,-2);

\draw[fill] (0,0) circle [radius=0.04];
\draw[fill] (0,0.2) circle [radius=0.04];
\draw[fill] (0,-0.2) circle [radius=0.04];

\draw[line width = \w mm] (-.25,-4-1.6) to [out=135,in=180] (0,-4+.25);
\draw[line width = 4*\w mm, white] (0,-2) to [out=180+45,in=180] (0,-.25-4);
\draw[line width = \w mm] (0,-2) to [out=180+45,in=180] (0,-.25-4);

\draw[line width = \w mm] (0,-.25-4) to [out=0, in=-45] (0,-2);
\draw[line width = 4*\w mm, white] (0,.25-4) to [out=0, in=45] (.25,-1.6-4);
\draw[line width = \w mm] (0,.25-4) to [out=0, in=45] (.25,-1.6-4);

\draw[fill, gray] (0,-6) circle [radius=1];
\draw[line width = \w mm] (0,-6) circle [radius=1];

\draw[fill] (0,5) circle [radius=0.15];
\draw[fill] (0,2) circle [radius=0.15];
\draw[fill] (0,-2) circle [radius=0.15];

\draw[line width = \w mm] (0,-2) to [out=180+35, in=180-65] (-1,-6);
\draw[line width = \w mm] (0,-2) to [out=-35, in=65] (1,-6);

\end{scope}

		\end{tikzpicture}\caption{On the left, a clasp Montesinos knot of the form $ M(0; \beta_1/\alpha_1, \dots, \beta_{r-2}/\alpha_{r-2}, 1/2, \beta_r/\alpha_r) $, where rational tangles are the gray circles and tunnels are represented by dashed lines. In the middle, the result of collapsing each tunnel. On the right, the diagram obtained by sliding the two outermost arcs from the top to the bottom vertex.}
		\label{fig:proof}
	\end{center}
\end{figure}
\end{proof}

\begin{corollary}\label{cor:nonaltclasp}
	In particular, if $ K $ is a clasp Montesinos knot and $ r = 3$, then $ t(K) = 1 $.
\end{corollary}

We also use part of Lustig and Moriah's theorem which gives the tunnel number of certain Montesinos knots.

\begin{theorem}[{\cite[Theorem 0.1]{LM}}]\label{thm:lustigmoriah}
	Let $ K $ be the Montesinos knot $ M(e; \beta_1/\alpha_1, \dots , \beta_r/\alpha_r) $, and let $ \alpha = \gcd(\alpha_1, \dots, \alpha_r) $. If $ \alpha \neq 1 $, then $ t(K) = b(K) -1 = r-1$.
\end{theorem}

Note that Theorem \ref{thm:lustigmoriah} does not contradict Proposition \ref{thm:clasp}, since any clasp Montesinos knot with $\alpha \not = 1$ will be a link with more than one component.

Now, we present Lackenby's main result.

\begin{theorem}[{\cite[Theorem 1]{Lackenby}}]\label{thm:Lackenby}
	An alternating knot $ K $ has tunnel number one if and only if $ K $ is a 2-bridge knot or $ K $ is a clasp Montesinos knot of the form $ M(e; \pm 1/2, \beta_1/\alpha_1, \beta_2/\alpha_2) $, where $ \alpha_1 $ and $ \alpha_2 $ are odd.
\end{theorem}

Combining Proposition \ref{thm:tunnelbridge} and Theorem \ref{thm:Lackenby}, we can compute the tunnel numbers for all alternating 3-bridge knots. Similarly, Proposition \ref{thm:clasp} and Theorem \ref{thm:Lackenby} give us a way to compute the tunnel number of many alternating 4-bridge knots.

\begin{corollary}\label{cor:3bridge}
Let $ K $ be an alternating knot.
\begin{itemize}
    \item If $K$ is 3-bridge and clasp Montesinos, then $t(K) = 1$.
    \item If $K$ is 3-bridge and not clasp Montesinos, then $t(K) = 2$.
    \item If $K$ is 4-bridge and clasp Montesinos, then $t(K) = 2$.
\end{itemize}
\end{corollary}

In contrast with Theorem \ref{thm:Lackenby} and Corollary \ref{cor:3bridge}, which exclusively apply to alternating knots, we have the following result by Morimoto, Sakuma, and Yokota, which completely characterizes tunnel number one Montesinos knots.

\begin{theorem}[{\cite[Theorem 2.2]{MSY2}}]\label{thm:MSYmontesinos} The Montesinos knot\footnote{In this paper, we exclude rational knots from the class of Montesinos knots. However, in \cite{MSY2}, rational knots are considered to be Montesinos knots, so this theorem has the additional condition that $ r =2 $.} $M(e;\beta_1 / \alpha_1, \beta_2/\alpha_2, \ldots, \beta_r/\alpha_r)$ has tunnel number one if and only if one of the following conditions holds up to cyclic permutation of the indices:
\begin{enumerate}
	\item $ r = 3 $, $ \alpha_{1} = 2 $, and $ \alpha_{2} \equiv \alpha_{3} \equiv 1 \ (\Mod 2) $.
	\item $ r = 3 $, $ \beta_2/\alpha_{2} \equiv \beta_3/\alpha_3 \in \mathbb{Q}/\Z $, and 
	\[ e - \sum_{i = 1}^{r} \beta_i/\alpha_i = \pm 1/(\alpha_{1}\alpha_{2}). \]
\end{enumerate}
\end{theorem}

Furthermore, the symmetry type of a knot can give an obstruction for a knot to have tunnel number one.

\begin{theorem}
[{\cite[Section 2.1]{Sakuma}}, {\cite[Theorem 1.2]{MSY2}}]
\label{thm:symmetry}
Any tunnel number one knot admits a strong inversion.
\end{theorem}

Additionally, there is a connection between the Nakanishi index $ m(K) $ of a knot $ K $ and its tunnel number $ t(K) $. The \emph{Nakanishi index} $ m(K) $ can be defined to be the minimal number of generators of the Alexander module of $ K $ \cite{Kawauchi}. We adopt the convention that knots with Alexander module isomorphic to $ \Z[t,t^{-1}] $ have Nakanishi index zero. As well, we have the following fact which appears as a footnote in Milnor's paper \cite{Milnor}: any 2-generator knot has a cyclic Alexander module. This, combined with Proposition \ref{thm:tunnelbridge}, directly implies the following:

\begin{proposition}\label{prop:Nakanishi}
	Any knot $ K $ with $ m(K) > 1$ has $ t(K) >1 $.
\end{proposition}

In \cite{Kohno}, Kohno uses quantum invariants to give some estimates for tunnel number. 

\begin{theorem}[\cite{Kohno}]\label{thm:Kohno}
Let $ K $ be a knot and let $ V_{K} $ denote its Jones polynomial. If $ K $ satisfies
\[ 
\abs{V_{K}\left(e^{2\pi \sqrt{-1}/5}\right)} > 2.1489,
 \]
then $ t(K) \geq 2$.
\end{theorem}

Finally, we present a nondeterministic method suggested to us by Nathan Dunfield for determining if a knot has tunnel number one.

\begin{method}\label{method}
Given a knot $ K $, use \snappy \ to find a presentation its knot group. Using Berge's \Heegaard \ program \cite{heegaard}, check whether this presentation comes from a Heegaard splitting by using the \texttt{is\_realizable} function. If this is the case and the knot group has two generators, then $ t(K) = 1 $. We may randomly re-triangulate the 3-manifold $ S^3\setminus N(K) $ many times using the \texttt{randomize} function in \snappy \ to obtain different knot group presentations.
\end{method} 

\section{Algorithms}\label{sec:main}

All the algorithms in this section have been implemented in \snappy ~by the first author. We will use $ RT(\ell) $ to denote the set of all fractions representing all rational tangles with $ \ell $ crossings. Each of the fractions in $ RT(\ell) $ is uniquely identified with a rational tangle. Step (1) of both algorithms lists all partitions which sum to $\ell$ or $n$, which is an easy exercise and we exclude it here.

\begin{algorithm}
	There exists an algorithm to identify all rational tangles with $ \ell $ crossings.
\end{algorithm}

The algorithm is the following:
	\begin{enumerate}
		\item List all integer partitions of $ \ell $, where each summand is a positive integer.
		\item Each partition $ \ell = a_{1} + a_{2} + \dots + a_{m} $ has an associated fraction defined as
		\[ \frac{p}{q} = a_{1} + \frac{1}{a_{2} + \dots + \frac{1}{a_{m-1}+\frac{1}{a_{m}}}}. \] For each partition $ [a_{1}, \dots , a_{n}] $, compute its associated fraction $ \frac{p}{q} $ and let $ X $ denote the set of all fractions arising this way. 
		\item For each fraction in $ X $, also add the negation of the fraction to $ X $.
	\end{enumerate}
	
Every rational tangle has an alternating minimal diagram \cite{KauffmanLambropoulou}, and every alternating diagram of a rational tangle created as in Figure \ref{fig:rationaltangle} is reduced and thus minimal \cite{KauffmanTait, MurasugiTait, ThistlethwaiteTait}, so we conclude that $ X = RT(\ell)$.

Given a fraction $ \frac{p}{q} $, we can use the function \texttt{RationalTangle(p,q)} from \snappy \ to build its corresponding rational tangle.

\begin{algorithm}\label{thm:montesinosalgo}
	There exists an algorithm to identify all Montesinos knots with $ n $ crossings.
\end{algorithm}

The main idea is to identify all Montesinos knots with $ n $ crossings and $ r $ rational tangles. Observe that $ r $ need be at most $ \lfloor n/2 \rfloor $, since each rational tangle must have at least two crossings. The algorithm is as follows:
\begin{enumerate}
	\item List all partitions of $ n $ with $ r $ summands.
	\item For each partition $ n  = n_{1} + \dots + n_{r}$, take the Cartesian product $ RT(n_{1}) \times \dots \times RT(n_{r}) $. Each element of this Cartesian product is an $ r $-tuple of rational numbers.
	\item For each such $ r $-tuple, construct a Montesinos link by tangle summing together all rational tangles from the tuple in order, and then taking the numerator closure of the sum. That is, for each fraction $ \frac{p_{i}}{q_{i}} $ in $ RT(n_{i}) $, form the Montesinos knot $ M(e; p_{1}/q_{1}, \dots , p_{r}/q_{r}) $.
	\item From this list of Montesinos links, check the number of components of each link by using the function \texttt{link\_components} from \snappy\ and remove all links with more than one component.
	\item Identify all Montesinos knots with $ n $ crossings by repeating this process for all $3 \leq r \leq \lfloor n/2 \rfloor $.
\end{enumerate}

Notice that in a Montesinos knot $ M(e; p_{1}/q_{1}, \dots , p_{r}/q_{r}) $, we can add the half-twists from $e$ to any of the rational tangles representing each fraction. At the same time, by \cite[Theorem 12.29]{BZ}, Montesinos knots with the same value of their fractions mod 1 are equivalent. Therefore, we may ignore $e$ in this algorithm.

We use the function \texttt{identify} from \snappy \ to identify knots and we then remove any possible duplicates.  In our work, every knot we created with this algorithm was identified by \snappy.

As a particular case of Algorithm \ref{thm:montesinosalgo}, we obtain the following result which will be later used to calculate tunnel numbers.

\begin{corollary}\label{thm:montesinos}
All 5525 Montesinos knots with 14 crossings or fewer have been identified. Moreover, all 2784 clasp Montesinos knots with 14 crossings or fewer have been identified.
\end{corollary}

The distribution of Montesinos knots per crossing number can be seen in Table \ref{table:Montesinos}.

\begin{table}[h]
\begin{center}
	\begin{tabular}{|c|c|c|c|c|c|c|}
		\hline
		\multirow{2}{*}{\begin{tabular}[c]{@{}c@{}}Number of \\ crossings\end{tabular}} & \multicolumn{3}{c|}{Number of Montesinos knots} & \multicolumn{3}{c|}{Number of clasp Montesinos knots} \\ \cline{2-7} 
		& Total     & Alternating    & Non-alternating    & Total       & Alternating      & Non-alternating      \\ \hline \hline
		$\leq$ 7                                                                           & 0         & 0              & 0                  & 0           & 0                & 0                    \\ \hline
		8                                                                                  & 6         & 6              & 0                  & 6           & 6                & 0                    \\ \hline
		9                                                                                  & 15         & 15              & 0                  & 11           & 11                & 0                    \\ \hline
		10                                                                                 & 57        & 57             & 0                  & 37           & 37                & 0                    \\ \hline
		11                                                                                 & 164       & 97             & 67                 & 101         & 60               & 41                   \\ \hline
		12                                                                                 & 479       & 283            & 196                & 265         & 159              & 106                  \\ \hline
		13                                                                                 & 1308      & 778            & 530                & 675         & 403              & 272                  \\ \hline
		14                                                                                 & 3496      & 2076           & 1420               & 1689        & 1004             & 685                  \\ \hline
		Total & 5525 & 3312 & 2213 & 2784 & 1680 & 1104 \\ \hline
	\end{tabular}
\caption{Number of Montesinos and clasp Montesinos knots per number of crossings.}\label{table:Montesinos}
\end{center}
\end{table}

At this moment, we are unable to list Montesinos knots with more than 14 crossings using our algorithm, since \snappy \ is currently unable to identify knots with more than 14 crossings.

\section{Tunnel numbers of 11 and 12 crossing knots}\label{sec:results}

We use KnotInfo \cite{knotinfo} to obtain the bridge number of all 11 or 12 crossing knots. Notice that the bridge number of all knots with 11 and 12 crossings is at most four.

\begin{proposition}\label{prop:alt}
	The tunnel number of all 1655 alternating 11 or 12 crossing knots has been calculated.
\end{proposition}

\begin{proof} Choose one of the 11 or 12 crossing alternating knots from the list. Let $ K $ be this knot. Using the bridge number data from KnotInfo \cite{knotinfo}, we have three cases.
	
\underline{Case 1:} If $ b(K) = 2 $, then $ t(K) = 1 $ by Theorem \ref{thm:Lackenby}.

\underline{Case 2:} Assume $ b(K) = 3 $. Using Corollary \ref{thm:montesinos}, we check if $ K $ is a clasp Montesinos knot, and then we apply Corollary \ref{cor:3bridge}. If $ K $ is a clasp Montesinos knot, then $ t(K) = 1 $. Otherwise, we have $ t(K) = 2$.

\underline{Case 3:} If $ b(K) = 4 $, we again apply Corollary \ref{thm:montesinos} to check whether $ K $ is a clasp Montesinos knot. Every knot in this case is a Montesinos knot. If $ K $ is a clasp Montesinos knot, then Corollary \ref{cor:3bridge} implies that it must have tunnel number two. This leaves exactly two knots, 12a0554 and 12a0750, which are non-clasp Montesinos knots. We now apply Theorem \ref{thm:lustigmoriah}. 12a0554 is the knot $ M(0; 2/3, 2/3, 2/3, 1/3) $ and 12a0750 is the knot $ M(0; 2/3, 1/3, 1/3, 1/3) $, both of which have $ \alpha = 3 $, hence, tunnel number three. 
\end{proof}
We obtain that there are 145 and 222 alternating 11 crossing knots with tunnel numbers one and two, respectively. For the alternating 12 crossing knots, there are 315, 971, and 2 knots with tunnel numbers one, two, and three, respectively. This information can be found in Table \ref{table:alternating}.
\begin{table}[h]
\begin{center}
\begin{scriptsize}
\begin{tabular}{|c|c|c|c|c|c|c|}
	\hline
	\multirow{13}{*}{\specialcell{Alternating:\\1655}} &
	\multirow{6}{*}{\specialcell{11 crossings:\\367}} &
	\multicolumn{4}{c|}{\specialcell{2-bridge:\\91}} &
	\specialcell{Tunnel number 1:\\91}\\
	\cline{3-7}
	
	& & \multirow{4}{*}{\specialcell{3-bridge:\\270}} &
	\multirow{3}{*}{\specialcell{Montesinos:\\91}} &
	\specialcell{Clasp:\\54} &
	\specialcell{$ \alpha = 1 $:\\54} &
	\specialcell{Tunnel number 1:\\54}\\
	\cline{5-7}
	
	& & & & \multirow{2}{*}{\specialcell{Non-clasp:\\37}} &
	\specialcell{$ \alpha = 1 $:\\35} &
	\multirow{2}{*}{\specialcell{Tunnel number 2:\\37}}\\
	\cline{6-6}
	
	& & & & & \specialcell{$ \alpha \neq 1 $:\\2} & \\
	\cline{4-7}
	
	& & & \multicolumn{3}{c|}{\specialcell{Non-Montesinos:\\179}} &
	\specialcell{Tunnel number 2:\\179}\\
	\cline{3-7}
	
	& & \specialcell{4-bridge:\\6} &
	\specialcell{Montesinos:\\6} &
	\specialcell{Clasp:\\6} &
	\specialcell{$ \alpha = 1 $:\\6} &
	\specialcell{Tunnel number 2:\\6}\\
	\cline{2-7}
	
	& \multirow{7}{*}{\specialcell{12 crossings:\\1288}} & \multicolumn{4}{c|}{\specialcell{2-bridge:\\176}} &
	\specialcell{Tunnel number 1:\\176}\\
	\cline{3-7}
	
	& & \multirow{4}{*}{\specialcell{3-bridge:\\1090} } &
	\multirow{3}{*}{\specialcell{Montesinos:\\261}} &
	\specialcell{Clasp:\\139} &
	\specialcell{$ \alpha = 1 $:\\139} &
	\specialcell{Tunnel number 1:\\139}\\
	\cline{5-7}
	
	& & & & \multirow{2}{*}{\specialcell{Non-clasp:\\122}} &
	\specialcell{$ \alpha = 1 $:\\100} &
	\multirow{2}{*}{\specialcell{Tunnel number 2:\\122}}\\
	\cline{6-6}
	
	& & & & & \specialcell{$ \alpha \neq 1 $:\\22} &\\
	\cline{4-7}
	
	& & & \multicolumn{3}{c|}{\specialcell{Non-Montesinos:\\829}} &
	\specialcell{Tunnel number 2:\\829}\\
	\cline{3-7}
	
	& & \multirow{2}{*}{\specialcell{4-bridge:\\22}} & 
	\multirow{2}{*}{\specialcell{Montesinos:\\22}} &
	\specialcell{Clasp:\\20} &
	\specialcell{$ \alpha = 1 $:\\20} & 
	\specialcell{Tunnel number 2:\\20}\\
	\cline{5-7}
	
	& & & & \specialcell{Non-clasp:\\2} &
	\specialcell{$ \alpha \neq 1 $:\\2} &
	\specialcell{Tunnel number 3:\\2}\\
	\hline
\end{tabular}
\end{scriptsize}
\caption{Identification of tunnel number for alternating knots with 11 and 12 crossings.}\label{table:alternating}
\end{center}
\end{table}

\begin{proposition}\label{prop:nonalt}
	The tunnel number of 881 non-alternating 11 or 12 crossing knots has been calculated.
\end{proposition}

\begin{proof}We follow the same procedure as in the proof of Proposition \ref{prop:alt}. In this case, we also need to employ Theorem \ref{thm:symmetry}, Proposition \ref{prop:Nakanishi}\footnote{The Nakanishi indices of knots with 10 crossings or fewer are known and we obtained them from KnotInfo \cite{knotinfo}. We obtained the Nakanishi indices of knots with 11 and 12 crossings from the Knotorious website \cite{knotorious}, but for many of these knots only upper and lower bounds for the Nakanishi index are known. By using Proposition \ref{prop:Nakanishi}, we are able to find an additional 82 knots with Nakanishi index 1 which were not identified by Knotorious.}, Theorem \ref{thm:Kohno}, and Method \ref{method}. Specific details can be found in Table \ref{table:nonalternating}.
\end{proof}

We have identified 144, 732, and 5 non-alternating 11 or 12 crossing knots with with tunnel numbers one, two, and three, respectively. This data can be found in Table \ref{table:nonalternating}.

\begin{table}[h]
	\begin{center}
		\begin{scriptsize}
			\begin{tabular}{|c|c|c|c|c|c|c|c|}
				\hline
				\multirow{16}{*}{\begin{tabular}[c]{@{}c@{}}Non-alternating:\\ 1073\end{tabular}} & \multirow{7}{*}{\begin{tabular}[c]{@{}c@{}}11 crossings:\\ 185\end{tabular}} & \multirow{6}{*}{\begin{tabular}[c]{@{}c@{}}3-bridge:\\ 176\end{tabular}} & \multirow{3}{*}{\begin{tabular}[c]{@{}c@{}}Montesinos:\\ 58\end{tabular}}  & \begin{tabular}[c]{@{}c@{}}Clasp:\\ 32\end{tabular}                      & \begin{tabular}[c]{@{}c@{}}$\alpha = 1$:\\ 32\end{tabular}                  & \begin{tabular}[c]{@{}c@{}}Tunnel number 1:\\ 32\end{tabular}  & \begin{tabular}[c]{@{}c@{}}Identified by:\\ \ref{cor:3bridge}, \ref{thm:MSYmontesinos}, \ref{method}\end{tabular}                                                                                     \\ \cline{5-8} 
				&                                                                              &                                                                          &                                                                            & \multirow{2}{*}{\begin{tabular}[c]{@{}c@{}}Non-clasp:\\ 26\end{tabular}} & \multirow{2}{*}{\begin{tabular}[c]{@{}c@{}}$\alpha = 1$:\\ 26\end{tabular}} & \begin{tabular}[c]{@{}c@{}}Tunnel number 1:\\ 2\end{tabular}   & \begin{tabular}[c]{@{}c@{}}Identified by:\\ \ref{thm:MSYmontesinos}, \ref{method}\end{tabular}                                                                                                                         \\ \cline{7-8} 
				&                                                                              &                                                                          &                                                                            &                                                                          &                                                                             & \begin{tabular}[c]{@{}c@{}}Tunnel number 2:\\ 24\end{tabular}  & \begin{tabular}[c]{@{}c@{}}Identified by:\\ \ref{thm:MSYmontesinos}, \ref{thm:symmetry}, \ref{prop:Nakanishi}\end{tabular}                                          \\ \cline{4-8} 
				&                                                                              &                                                                          & \multicolumn{3}{c|}{\multirow{3}{*}{\begin{tabular}[c]{@{}c@{}}Non-Montesinos:\\ 118\end{tabular}}}                                                                                                                                 & \begin{tabular}[c]{@{}c@{}}Tunnel number 1:\\ 5\end{tabular}   & \begin{tabular}[c]{@{}c@{}}Identified by:\\ \ref{method}\end{tabular}                                                                                                                                                                   \\ \cline{7-8} 
				&                                                                              &                                                                          & \multicolumn{3}{c|}{}                                                                                                                                                                                                               & \begin{tabular}[c]{@{}c@{}}Tunnel number 2:\\ 71\end{tabular}  & \begin{tabular}[c]{@{}c@{}}Identified by:\\ \ref{thm:symmetry}, \ref{prop:Nakanishi}, \ref{thm:Kohno}\end{tabular}                                                                                    \\ \cline{7-8} 
				&                                                                              &                                                                          & \multicolumn{3}{c|}{}                                                                                                                                                                                                               & \multicolumn{2}{c|}{\begin{tabular}[c]{@{}c@{}}Tunnel number $ \in \brackets{1,2}$:\\ 42\end{tabular}}                                                                                                                                                                                                                    \\ \cline{3-8} 
				&                                                                              & \begin{tabular}[c]{@{}c@{}}4-bridge:\\ 9\end{tabular}                    & \begin{tabular}[c]{@{}c@{}}Montesinos:\\ 9\end{tabular}                    & \begin{tabular}[c]{@{}c@{}}Clasp:\\ 9\end{tabular}                       & \begin{tabular}[c]{@{}c@{}}$\alpha = 1$:\\ 9\end{tabular}                   & \begin{tabular}[c]{@{}c@{}}Tunnel number 2:\\ 9\end{tabular}   & \begin{tabular}[c]{@{}c@{}}Identified by:\\ \ref{thm:MSYmontesinos}, \ref{prop:Nakanishi}\end{tabular}                                          \\ \cline{2-8} 
				& \multirow{9}{*}{\begin{tabular}[c]{@{}c@{}}12 crossings:\\ 888\end{tabular}} & \multirow{7}{*}{\begin{tabular}[c]{@{}c@{}}3-bridge:\\ 862\end{tabular}} & \multirow{4}{*}{\begin{tabular}[c]{@{}c@{}}Montesinos:\\ 170\end{tabular}} & \begin{tabular}[c]{@{}c@{}}Clasp:\\ 85\end{tabular}                      & \begin{tabular}[c]{@{}c@{}}$\alpha = 1$:\\ 85\end{tabular}                  & \begin{tabular}[c]{@{}c@{}}Tunnel number 1:\\ 85\end{tabular}  & \begin{tabular}[c]{@{}c@{}}Identified by:\\ \ref{cor:3bridge}, \ref{thm:MSYmontesinos}, \ref{method}\end{tabular}                                                                                     \\ \cline{5-8} 
				&                                                                              &                                                                          &                                                                            & \multirow{3}{*}{\begin{tabular}[c]{@{}c@{}}Non-clasp:\\ 85\end{tabular}} & \multirow{2}{*}{\begin{tabular}[c]{@{}c@{}}$\alpha = 1$:\\ 65\end{tabular}} & \begin{tabular}[c]{@{}c@{}}Tunnel number 1:\\ 2\end{tabular}   & \begin{tabular}[c]{@{}c@{}}Identified by:\\ \ref{thm:MSYmontesinos}, \ref{method}\end{tabular}                                                                                                                         \\ \cline{7-8} 
				&                                                                              &                                                                          &                                                                            &                                                                          &                                                                             & \begin{tabular}[c]{@{}c@{}}Tunnel number 2:\\ 63\end{tabular}  & \begin{tabular}[c]{@{}c@{}}Identified by:\\ \ref{thm:MSYmontesinos}, \ref{thm:symmetry}, \ref{prop:Nakanishi}\end{tabular}                                          \\ \cline{6-8} 
				&                                                                              &                                                                          &                                                                            &                                                                          & \begin{tabular}[c]{@{}c@{}}$\alpha \neq 1$:\\ 20\end{tabular}               & \begin{tabular}[c]{@{}c@{}}Tunnel number 2:\\ 20\end{tabular}  & \begin{tabular}[c]{@{}c@{}}Identified by:\\ \ref{thm:lustigmoriah}, \ref{thm:MSYmontesinos}, \ref{thm:symmetry}, \ref{prop:Nakanishi}\end{tabular} \\ \cline{4-8} 
				&                                                                              &                                                                          & \multicolumn{3}{c|}{\multirow{3}{*}{\begin{tabular}[c]{@{}c@{}}Non-Montesinos:\\ 692\end{tabular}}}                                                                                                                                 & \begin{tabular}[c]{@{}c@{}}Tunnel number 1:\\ 18\end{tabular}  & \begin{tabular}[c]{@{}c@{}}Identified by:\\ \ref{method}\end{tabular}                                                                                                                                                                   \\ \cline{7-8} 
				&                                                                              &                                                                          & \multicolumn{3}{c|}{}                                                                                                                                                                                                               & \begin{tabular}[c]{@{}c@{}}Tunnel number 2:\\ 524\end{tabular} & \begin{tabular}[c]{@{}c@{}}Identified by:\\ \ref{thm:symmetry}, \ref{prop:Nakanishi}, \ref{thm:Kohno}\end{tabular}                                                                                    \\ \cline{7-8} 
				&                                                                              &                                                                          & \multicolumn{3}{c|}{}                                                                                                                                                                                                               & \multicolumn{2}{c|}{\begin{tabular}[c]{@{}c@{}}Tunnel number $\in \brackets{1,2}$:\\ 150\end{tabular}}                                                                                                                                                                                                                    \\ \cline{3-8} 
				&                                                                              & \multirow{2}{*}{\begin{tabular}[c]{@{}c@{}}4-bridge:\\ 26\end{tabular}}  & \multirow{2}{*}{\begin{tabular}[c]{@{}c@{}}Montesinos:\\ 26\end{tabular}}  & \begin{tabular}[c]{@{}c@{}}Clasp:\\ 21\end{tabular}                      & \begin{tabular}[c]{@{}c@{}}$\alpha = 1$:\\ 21\end{tabular}                  & \begin{tabular}[c]{@{}c@{}}Tunnel number 2:\\ 21\end{tabular}  & \begin{tabular}[c]{@{}c@{}}Identified by:\\ \ref{thm:MSYmontesinos}, \ref{prop:Nakanishi}\end{tabular}                                          \\ \cline{5-8} 
				&                                                                              &                                                                          &                                                                            & \begin{tabular}[c]{@{}c@{}}Non-clasp:\\ 5\end{tabular}                   & \begin{tabular}[c]{@{}c@{}}$\alpha \neq 1$:\\ 5\end{tabular}                & \begin{tabular}[c]{@{}c@{}}Tunnel number 3:\\ 5\end{tabular}   & \begin{tabular}[c]{@{}c@{}}Identified by:\\ \ref{thm:lustigmoriah}\end{tabular}                                                                                                                                                         \\ \hline
			\end{tabular}
		\end{scriptsize}
		\caption{Identification of tunnel number for non-alternating knots with 11 and 12 crossings.}\label{table:nonalternating}
	\end{center}
\end{table}

Table \ref{table:percent} shows the relative effectiveness of each criterion in computing tunnel number. In particular, Method \ref{method} was able to identify all 756 tunnel number one knots known to us; this includes 23 non-alternating tunnel number one knots which were not identified by any other criterion.

\begin{table}[h]
	\begin{center}
	\begin{normalsize}
	\begin{tabular}{|c|c|c|}
		\hline
		Criteria                                      & \begin{tabular}[c]{@{}c@{}}Amount of tunnel\\ numbers identified \\ with this criterion\end{tabular} & \begin{tabular}[c]{@{}c@{}}\% of tunnel \\ numbers identified \\ with this criterion\end{tabular} \\ \hline \hline
		Corollary \ref{cor:nonaltclasp}                                 & 364                                                                                                & 14.35\%                                                                     \\ \hline
		Theorem \ref{thm:lustigmoriah}                                     & 57                                                                                                 & 2.24\%                                                                      \\ \hline
		Theorem \ref{thm:Lackenby}, Corollary \ref{cor:3bridge} & 1903                                                                                               & 75.04\%                                                                     \\ \hline
		Theorem \ref{thm:MSYmontesinos}                                    & 714                                                                                                & 28.15\%                                                                     \\ \hline
		Theorem \ref{thm:symmetry}                                   & 1366                                                                                               & 53.86\%                                                                     \\ \hline
		Proposition \ref{prop:Nakanishi}                        & 362                                                                                                & 14.27\%                                                                     \\ \hline
		Theorem \ref{thm:Kohno}                                  & 89                                                                                                 & 3.51\%                                                                      \\ \hline \hline
		Total amount of tunnel numbers identified      & 2536                                                                                               & 100\%                                                                      \\ \hline
	\end{tabular}
\end{normalsize}
\caption{Relative effectiveness of each criterion in computing tunnel number, out of the 2536 tunnel numbers that we know of knots with 12 crossings or fewer. Propositions \ref{thm:tunnelbridge} and \ref{thm:clasp} are used implicitly.}\label{table:percent}
\end{center}
\end{table}

We have been unable to identify the tunnel number of 192 of the non-altenating knots with 11 and 12 crossings. The tunnel number of these 192 knots can be bounded above by Theorem \ref{thm:tunnelbridge} and Corollary \ref{cor:nonaltclasp}, so these knots have tunnel number at most two. Based on the effectiveness of Method \ref{method}, we conjecture that all of these 192 knots have tunnel number two.

Overall, of the 2728 total knots with 11 and 12 crossings, we have found 2536 tunnel numbers. Combining Propositions \ref{prop:alt} and \ref{prop:nonalt}, we obtain Theorem \ref{thm:main}.

\section{Future directions}

Some alternative methods to compute tunnel number not used in this paper include computing Yamada's invariant of spatial graphs \cite{MSY2}, finding better methods to compute the Nakanishi indices of knots, or implementing the algorithms by Lackenby \cite{LackenbyAlgorithm} or Li \cite{LiAlgorithm}.

A possible continuation of this work includes computing the tunnel numbers of knots with 13 or more crossings. In this case, it would be useful to have a list of Nakanishi indices of knots with 13 or more crossings (possibly computed using the Knotorious program \cite{knotorious}). The bridge indices of knots up to 16 crossings can be found from \cite{BKVV} and the list of Montesinos knots up to 14 knots can be found from this paper.

We hope that our tunnel number data can be used to make a list of bridge spectra of knots \cite{Zupan}, to help conjecture a full characterization of tunnel number two knots, or as evidence for or against the Rank-Genus conjecture for knots \cite[Question 2]{Li}.

\appendix
\section{List of tunnel numbers}\label{appendix}

\noindent \textbf{11 and 12 crossing knots with tunnel number one:}

\noindent \tiny{11a$ X $ with $X  = $ 1,
	4,
	5,
	6,
	7,
	8,
	9,
	10,
	11,
	12,
	13,
	16,
	21,
	23,
	31,
	32,
	33,
	34,
	35,
	36,
	37,
	39,
	40,
	41,
	42,
	45,
	46,
	48,
	50,
	51,
	55,
	56,
	58,
	59,
	60,
	61,
	62,
	63,
	64,
	65,
	74,
	75,
	77,
	82,
	83,
	84,
	85,
	89,
	90,
	91,
	92,
	93,
	94,
	95,
	96,
	98,
	110,
	111,
	117,
	118,
	119,
	120,
	121,
	140,
	144,
	145,
	153,
	154,
	159,
	161,
	166,
	174,
	175,
	176,
	177,
	178,
	179,
	180,
	182,
	183,
	184,
	185,
	186,
	188,
	190,
	191,
	192,
	193,
	195,
	203,
	204,
	205,
	206,
	207,
	208,
	210,
	211,
	220,
	221,
	222,
	224,
	225,
	226,
	229,
	230,
	234,
	235,
	236,
	238,
	240,
	241,
	242,
	243,
	245,
	246,
	247,
	257,
	258,
	259,
	260,
	306,
	307,
	308,
	309,
	310,
	311,
	333,
	334,
	335,
	336,
	337,
	339,
	341,
	342,
	343,
	355,
	356,
	357,
	358,
	359,
	360,
	363,
	364,
	365,
	367.}

\noindent \tiny{11n$ X $ with $X = $ 1,
	2,
	3,
	12,
	13,
	14,
	15,
	16,
	17,
	18,
	19,
	20,
	28,
	29,
	30,
	38,
	51,
	52,
	53,
	54,
	55,
	56,
	57,
	58,
	59,
	60,
	61,
	62,
	63,
	64,
	70,
	79,
	96,
	102,
	104,
	111,
	135,
	143,
	145.}
	
\noindent \tiny{12a$ X $ with $X = $ 3,
9,
12,
16,
17,
18,
19,
20,
21,
22,
24,
25,
26,
27,
28,
31,
32,
34,
35,
37,
38,
42,
56,
62,
77,
78,
81,
82,
83,
84,
85,
86,
87,
95,
96,
97,
98,
99,
104,
106,
110,
112,
118,
121,
124,
128,
130,
141,
142,
143,
144,
145,
146,
147,
148,
149,
151,
152,
153,
158,
159,
160,
161,
165,
168,
169,
170,
171,
172,
173,
174,
175,
176,
178,
179,
180,
181,
183,
193,
194,
196,
197,
203,
204,
205,
206,
210,
212,
221,
226,
234,
235,
236,
237,
238,
239,
240,
241,
242,
243,
246,
247,
251,
254,
255,
257,
258,
259,
299,
300,
302,
303,
304,
305,
306,
307,
329,
330,
345,
378,
379,
380,
384,
385,
406,
420,
421,
422,
423,
424,
425,
436,
437,
447,
454,
463,
471,
477,
482,
497,
498,
499,
500,
501,
502,
506,
507,
508,
510,
511,
512,
514,
517,
518,
519,
520,
521,
522,
528,
532,
533,
534,
535,
536,
537,
538,
539,
540,
541,
544,
545,
549,
550,
551,
552,
579,
580,
581,
582,
583,
584,
585,
591,
595,
596,
597,
600,
601,
643,
644,
649,
650,
651,
652,
669,
681,
682,
684,
689,
690,
691,
713,
714,
715,
716,
717,
718,
720,
721,
722,
723,
724,
726,
727,
728,
729,
731,
732,
733,
736,
738,
740,
743,
744,
745,
757,
758,
759,
760,
761,
762,
763,
764,
772,
773,
774,
775,
789,
790,
791,
792,
794,
795,
796,
797,
800,
802,
803,
822,
823,
826,
827,
835,
836,
837,
838,
839,
840,
842,
843,
876,
877,
878,
879,
880,
881,
882,
937,
938,
1023,
1024,
1029,
1030,
1033,
1034,
1039,
1040,
1125,
1126,
1127,
1128,
1129,
1130,
1131,
1132,
1133,
1134,
1135,
1136,
1138,
1139,
1140,
1145,
1146,
1148,
1149,
1157,
1158,
1159,
1161,
1162,
1163,
1165,
1166,
1214,
1273,
1274,
1275,
1276,
1277,
1278,
1279,
1281,
1282,
1287.}

\noindent \tiny{12n$ X $ with $ X =  $ 11,
	12,
	13,
	25,
	35,
	36,
	37,
	38,
	39,
	40,
	41,
	42,
	43,
	44,
	45,
	46,
	47,
	48,
	54,
	65,
	77,
	78,
	79,
	121,
	150,
	151,
	152,
	153,
	154,
	155,
	159,
	160,
	161,
	162,
	163,
	164,
	165,
	166,
	167,
	168,
	169,
	170,
	171,
	172,
	198,
	199,
	200,
	218,
	233,
	234,
	235,
	236,
	237,
	238,
	239,
	240,
	241,
	242,
	243,
	244,
	248,
	249,
	250,
	251,
	288,
	289,
	293,
	303,
	304,
	305,
	306,
	307,
	308,
	309,
	310,
	352,
	370,
	371,
	374,
	404,
	433,
	446,
	464,
	483,
	487,
	488,
	500,
	501,
	502,
	503,
	552,
	575,
	579,
	591,
	594,
	624,
	650,
	721,
	722,
	723,
	724,
	725,
	726,
	749,
	851.}

\normalsize

\noindent \textbf{11 and 12 crossing knots with tunnel number two:}

\noindent \tiny{11a$ X $ with $ X= $ 2,
	3,
	14,
	15,
	17,
	18,
	19,
	20,
	22,
	24,
	25,
	26,
	27,
	28,
	29,
	30,
	38,
	43,
	44,
	47,
	49,
	52,
	53,
	54,
	57,
	66,
	67,
	68,
	69,
	70,
	71,
	72,
	73,
	76,
	78,
	79,
	80,
	81,
	86,
	87,
	88,
	97,
	99,
	100,
	101,
	102,
	103,
	104,
	105,
	106,
	107,
	108,
	109,
	112,
	113,
	114,
	115,
	116,
	122,
	123,
	124,
	125,
	126,
	127,
	128,
	129,
	130,
	131,
	132,
	133,
	134,
	135,
	136,
	137,
	138,
	139,
	141,
	142,
	143,
	146,
	147,
	148,
	149,
	150,
	151,
	152,
	155,
	156,
	157,
	158,
	160,
	162,
	163,
	164,
	165,
	167,
	168,
	169,
	170,
	171,
	172,
	173,
	181,
	187,
	189,
	194,
	196,
	197,
	198,
	199,
	200,
	201,
	202,
	209,
	212,
	213,
	214,
	215,
	216,
	217,
	218,
	219,
	223,
	227,
	228,
	231,
	232,
	233,
	237,
	239,
	244,
	248,
	249,
	250,
	251,
	252,
	253,
	254,
	255,
	256,
	261,
	262,
	263,
	264,
	265,
	266,
	267,
	268,
	269,
	270,
	271,
	272,
	273,
	274,
	275,
	276,
	277,
	278,
	279,
	280,
	281,
	282,
	283,
	284,
	285,
	286,
	287,
	288,
	289,
	290,
	291,
	292,
	293,
	294,
	295,
	296,
	297,
	298,
	299,
	300,
	301,
	302,
	303,
	304,
	305,
	312,
	313,
	314,
	315,
	316,
	317,
	318,
	319,
	320,
	321,
	322,
	323,
	324,
	325,
	326,
	327,
	328,
	329,
	330,
	331,
	332,
	338,
	340,
	344,
	345,
	346,
	347,
	348,
	349,
	350,
	351,
	352,
	353,
	354,
	361,
	362,
	366.}

\noindent \tiny{11n$ X $ with $ X= $ 4,
	5,
	6,
	7,
	8,
	9,
	10,
	11,
	21,
	22,
	23,
	24,
	25,
	26,
	27,
	31,
	32,
	33,
	34,
	35,
	36,
	37,
	39,
	40,
	41,
	42,
	43,
	44,
	46,
	47,
	48,
	49,
	50,
	67,
	68,
	69,
	71,
	72,
	73,
	74,
	75,
	76,
	77,
	78,
	80,
	81,
	82,
	83,
	84,
	85,
	87,
	88,
	89,
	90,
	91,
	93,
	97,
	100,
	101,
	103,
	105,
	106,
	107,
	108,
	109,
	110,
	114,
	116,
	120,
	122,
	124,
	126,
	128,
	129,
	130,
	131,
	132,
	133,
	134,
	137,
	138,
	139,
	140,
	141,
	147,
	148,
	151,
	154,
	157,
	159,
	160,
	162,
	164,
	165,
	166,
	167,
	172,
	174,
	175,
	176,
	177,
	183,
	184,
	185.}

\noindent \tiny{12a$ X $ with $ X= $ 1,
	2,
	4,
	5,
	6,
	7,
	8,
	10,
	11,
	13,
	14,
	15,
	23,
	29,
	30,
	33,
	36,
	39,
	40,
	41,
	43,
	44,
	45,
	46,
	47,
	48,
	49,
	50,
	51,
	52,
	53,
	54,
	55,
	57,
	58,
	59,
	60,
	61,
	63,
	64,
	65,
	66,
	67,
	68,
	69,
	70,
	71,
	72,
	73,
	74,
	75,
	76,
	79,
	80,
	88,
	89,
	90,
	91,
	92,
	93,
	94,
	100,
	101,
	102,
	103,
	105,
	107,
	108,
	109,
	111,
	113,
	114,
	115,
	116,
	117,
	119,
	120,
	122,
	123,
	125,
	126,
	127,
	129,
	131,
	132,
	133,
	134,
	135,
	136,
	137,
	138,
	139,
	140,
	150,
	154,
	155,
	156,
	157,
	162,
	163,
	164,
	166,
	167,
	177,
	182,
	184,
	185,
	186,
	187,
	188,
	189,
	190,
	191,
	192,
	195,
	198,
	199,
	200,
	201,
	202,
	207,
	208,
	209,
	211,
	213,
	214,
	215,
	216,
	217,
	218,
	219,
	220,
	222,
	223,
	224,
	225,
	227,
	228,
	229,
	230,
	231,
	232,
	233,
	244,
	245,
	248,
	249,
	250,
	252,
	253,
	256,
	260,
	261,
	262,
	263,
	264,
	265,
	266,
	267,
	268,
	269,
	270,
	271,
	272,
	273,
	274,
	275,
	276,
	277,
	278,
	279,
	280,
	281,
	282,
	283,
	284,
	285,
	286,
	287,
	288,
	289,
	290,
	291,
	292,
	293,
	294,
	295,
	296,
	297,
	298,
	301,
	308,
	309,
	310,
	311,
	312,
	313,
	314,
	315,
	316,
	317,
	318,
	319,
	320,
	321,
	322,
	323,
	324,
	325,
	326,
	327,
	328,
	331,
	332,
	333,
	334,
	335,
	336,
	337,
	338,
	339,
	340,
	341,
	342,
	343,
	344,
	346,
	347,
	348,
	349,
	350,
	351,
	352,
	353,
	354,
	355,
	356,
	357,
	358,
	359,
	360,
	361,
	362,
	363,
	364,
	365,
	366,
	367,
	368,
	369,
	370,
	371,
	372,
	373,
	374,
	375,
	376,
	377,
	381,
	382,
	383,
	386,
	387,
	388,
	389,
	390,
	391,
	392,
	393,
	394,
	395,
	396,
	397,
	398,
	399,
	400,
	401,
	402,
	403,
	404,
	405,
	407,
	408,
	409,
	410,
	411,
	412,
	413,
	414,
	415,
	416,
	417,
	418,
	419,
	426,
	427,
	428,
	429,
	430,
	431,
	432,
	433,
	434,
	435,
	438,
	439,
	440,
	441,
	442,
	443,
	444,
	445,
	446,
	448,
	449,
	450,
	451,
	452,
	453,
	455,
	456,
	457,
	458,
	459,
	460,
	461,
	462,
	464,
	465,
	466,
	467,
	468,
	469,
	470,
	472,
	473,
	474,
	475,
	476,
	478,
	479,
	480,
	481,
	483,
	484,
	485,
	486,
	487,
	488,
	489,
	490,
	491,
	492,
	493,
	494,
	495,
	496,
	503,
	504,
	505,
	509,
	513,
	515,
	516,
	523,
	524,
	525,
	526,
	527,
	529,
	530,
	531,
	542,
	543,
	546,
	547,
	548,
	553,
	555,
	556,
	557,
	558,
	559,
	560,
	561,
	562,
	563,
	564,
	565,
	566,
	567,
	568,
	569,
	570,
	571,
	572,
	573,
	574,
	575,
	576,
	577,
	578,
	586,
	587,
	588,
	589,
	590,
	592,
	593,
	594,
	598,
	599,
	602,
	603,
	604,
	605,
	606,
	607,
	608,
	609,
	610,
	611,
	612,
	613,
	614,
	615,
	616,
	617,
	618,
	619,
	620,
	621,
	622,
	623,
	624,
	625,
	626,
	627,
	628,
	629,
	630,
	631,
	632,
	633,
	634,
	635,
	636,
	637,
	638,
	639,
	640,
	641,
	642,
	645,
	646,
	647,
	648,
	653,
	654,
	655,
	656,
	657,
	658,
	659,
	660,
	661,
	662,
	663,
	664,
	665,
	666,
	667,
	668,
	670,
	671,
	672,
	673,
	674,
	675,
	676,
	677,
	678,
	679,
	680,
	683,
	685,
	686,
	687,
	688,
	692,
	693,
	694,
	695,
	696,
	697,
	698,
	699,
	700,
	701,
	702,
	703,
	704,
	705,
	706,
	707,
	708,
	709,
	710,
	711,
	712,
	719,
	725,
	730,
	734,
	735,
	737,
	739,
	741,
	742,
	746,
	747,
	748,
	749,
	751,
	752,
	753,
	754,
	755,
	756,
	765,
	766,
	767,
	768,
	769,
	770,
	771,
	776,
	777,
	778,
	779,
	780,
	781,
	782,
	783,
	784,
	785,
	786,
	787,
	788,
	793,
	798,
	799,
	801,
	804,
	805,
	806,
	807,
	808,
	809,
	810,
	811,
	812,
	813,
	814,
	815,
	816,
	817,
	818,
	819,
	820,
	821,
	824,
	825,
	828,
	829,
	830,
	831,
	832,
	833,
	834,
	841,
	844,
	845,
	846,
	847,
	848,
	849,
	850,
	851,
	852,
	853,
	854,
	855,
	856,
	857,
	858,
	859,
	860,
	861,
	862,
	863,
	864,
	865,
	866,
	867,
	868,
	869,
	870,
	871,
	872,
	873,
	874,
	875,
	883,
	884,
	885,
	886,
	887,
	888,
	889,
	890,
	891,
	892,
	893,
	894,
	895,
	896,
	897,
	898,
	899,
	900,
	901,
	902,
	903,
	904,
	905,
	906,
	907,
	908,
	909,
	910,
	911,
	912,
	913,
	914,
	915,
	916,
	917,
	918,
	919,
	920,
	921,
	922,
	923,
	924,
	925,
	926,
	927,
	928,
	929,
	930,
	931,
	932,
	933,
	934,
	935,
	936,
	939,
	940,
	941,
	942,
	943,
	944,
	945,
	946,
	947,
	948,
	949,
	950,
	951,
	952,
	953,
	954,
	955,
	956,
	957,
	958,
	959,
	960,
	961,
	962,
	963,
	964,
	965,
	966,
	967,
	968,
	969,
	970,
	971,
	972,
	973,
	974,
	975,
	976,
	977,
	978,
	979,
	980,
	981,
	982,
	983,
	984,
	985,
	986,
	987,
	988,
	989,
	990,
	991,
	992,
	993,
	994,
	995,
	996,
	997,
	998,
	999,
	1000,
	1001,
	1002,
	1003,
	1004,
	1005,
	1006,
	1007,
	1008,
	1009,
	1010,
	1011,
	1012,
	1013,
	1014,
	1015,
	1016,
	1017,
	1018,
	1019,
	1020,
	1021,
	1022,
	1025,
	1026,
	1027,
	1028,
	1031,
	1032,
	1035,
	1036,
	1037,
	1038,
	1041,
	1042,
	1043,
	1044,
	1045,
	1046,
	1047,
	1048,
	1049,
	1050,
	1051,
	1052,
	1053,
	1054,
	1055,
	1056,
	1057,
	1058,
	1059,
	1060,
	1061,
	1062,
	1063,
	1064,
	1065,
	1066,
	1067,
	1068,
	1069,
	1070,
	1071,
	1072,
	1073,
	1074,
	1075,
	1076,
	1077,
	1078,
	1079,
	1080,
	1081,
	1082,
	1083,
	1084,
	1085,
	1086,
	1087,
	1088,
	1089,
	1090,
	1091,
	1092,
	1093,
	1094,
	1095,
	1096,
	1097,
	1098,
	1099,
	1100,
	1101,
	1102,
	1103,
	1104,
	1105,
	1106,
	1107,
	1108,
	1109,
	1110,
	1111,
	1112,
	1113,
	1114,
	1115,
	1116,
	1117,
	1118,
	1119,
	1120,
	1121,
	1122,
	1123,
	1124,
	1137,
	1141,
	1142,
	1143,
	1144,
	1147,
	1150,
	1151,
	1152,
	1153,
	1154,
	1155,
	1156,
	1160,
	1164,
	1167,
	1168,
	1169,
	1170,
	1171,
	1172,
	1173,
	1174,
	1175,
	1176,
	1177,
	1178,
	1179,
	1180,
	1181,
	1182,
	1183,
	1184,
	1185,
	1186,
	1187,
	1188,
	1189,
	1190,
	1191,
	1192,
	1193,
	1194,
	1195,
	1196,
	1197,
	1198,
	1199,
	1200,
	1201,
	1202,
	1203,
	1204,
	1205,
	1206,
	1207,
	1208,
	1209,
	1210,
	1211,
	1212,
	1213,
	1215,
	1216,
	1217,
	1218,
	1219,
	1220,
	1221,
	1222,
	1223,
	1224,
	1225,
	1226,
	1227,
	1228,
	1229,
	1230,
	1231,
	1232,
	1233,
	1234,
	1235,
	1236,
	1237,
	1238,
	1239,
	1240,
	1241,
	1242,
	1243,
	1244,
	1245,
	1246,
	1247,
	1248,
	1249,
	1250,
	1251,
	1252,
	1253,
	1254,
	1255,
	1256,
	1257,
	1258,
	1259,
	1260,
	1261,
	1262,
	1263,
	1264,
	1265,
	1266,
	1267,
	1268,
	1269,
	1270,
	1271,
	1272,
	1280,
	1283,
	1284,
	1285,
	1286,
	1288.}

\noindent \tiny{12n$ X $ with $ X= $ 1,
	2,
	3,
	4,
	5,
	6,
	7,
	8,
	9,
	10,
	14,
	15,
	17,
	18,
	19,
	20,
	21,
	22,
	23,
	24,
	26,
	27,
	28,
	29,
	30,
	31,
	32,
	33,
	34,
	49,
	50,
	51,
	55,
	56,
	57,
	58,
	59,
	60,
	61,
	62,
	63,
	64,
	66,
	67,
	68,
	69,
	70,
	71,
	72,
	73,
	74,
	75,
	76,
	80,
	81,
	82,
	83,
	84,
	85,
	86,
	87,
	89,
	90,
	91,
	92,
	93,
	94,
	95,
	96,
	97,
	98,
	99,
	100,
	101,
	102,
	103,
	104,
	105,
	106,
	107,
	108,
	109,
	110,
	111,
	112,
	113,
	114,
	115,
	116,
	117,
	118,
	119,
	120,
	122,
	123,
	124,
	125,
	126,
	127,
	128,
	129,
	130,
	131,
	132,
	133,
	134,
	135,
	136,
	137,
	138,
	139,
	140,
	141,
	142,
	143,
	144,
	145,
	146,
	147,
	148,
	149,
	173,
	174,
	175,
	176,
	177,
	178,
	179,
	180,
	181,
	182,
	183,
	184,
	185,
	186,
	187,
	188,
	189,
	190,
	191,
	192,
	193,
	194,
	195,
	196,
	197,
	201,
	202,
	203,
	205,
	206,
	208,
	209,
	210,
	212,
	213,
	214,
	215,
	216,
	217,
	219,
	220,
	221,
	222,
	223,
	224,
	225,
	226,
	227,
	229,
	231,
	232,
	252,
	253,
	254,
	255,
	256,
	257,
	259,
	261,
	262,
	263,
	264,
	265,
	266,
	267,
	268,
	269,
	270,
	271,
	272,
	273,
	274,
	275,
	276,
	277,
	278,
	281,
	286,
	290,
	291,
	292,
	294,
	295,
	296,
	297,
	298,
	300,
	301,
	302,
	312,
	313,
	315,
	316,
	317,
	318,
	319,
	320,
	322,
	323,
	324,
	325,
	326,
	327,
	331,
	332,
	333,
	334,
	335,
	336,
	337,
	338,
	339,
	340,
	341,
	342,
	343,
	344,
	345,
	346,
	347,
	348,
	349,
	350,
	351,
	353,
	355,
	356,
	357,
	358,
	359,
	360,
	361,
	362,
	363,
	364,
	365,
	366,
	367,
	368,
	369,
	372,
	376,
	378,
	379,
	380,
	381,
	382,
	384,
	385,
	386,
	387,
	388,
	389,
	390,
	393,
	394,
	396,
	397,
	398,
	399,
	400,
	401,
	402,
	403,
	405,
	406,
	407,
	408,
	409,
	410,
	412,
	413,
	414,
	415,
	418,
	419,
	420,
	421,
	422,
	423,
	424,
	427,
	428,
	429,
	430,
	431,
	434,
	435,
	436,
	437,
	440,
	442,
	444,
	447,
	448,
	454,
	455,
	456,
	457,
	460,
	461,
	462,
	463,
	465,
	466,
	467,
	468,
	469,
	470,
	471,
	472,
	473,
	474,
	475,
	476,
	477,
	478,
	479,
	480,
	484,
	485,
	491,
	492,
	493,
	494,
	495,
	496,
	497,
	498,
	504,
	505,
	506,
	507,
	508,
	509,
	510,
	512,
	513,
	514,
	515,
	516,
	517,
	518,
	520,
	522,
	523,
	526,
	528,
	529,
	530,
	531,
	532,
	533,
	534,
	536,
	539,
	540,
	541,
	545,
	546,
	547,
	549,
	550,
	557,
	558,
	559,
	560,
	561,
	562,
	563,
	564,
	565,
	566,
	567,
	568,
	569,
	570,
	571,
	572,
	573,
	574,
	576,
	577,
	578,
	581,
	582,
	583,
	584,
	586,
	588,
	589,
	590,
	592,
	593,
	597,
	598,
	599,
	600,
	601,
	602,
	604,
	605,
	606,
	607,
	608,
	609,
	611,
	614,
	616,
	617,
	618,
	619,
	620,
	621,
	622,
	623,
	626,
	627,
	630,
	631,
	633,
	634,
	635,
	636,
	637,
	640,
	641,
	643,
	644,
	645,
	646,
	647,
	648,
	649,
	651,
	652,
	654,
	655,
	656,
	657,
	658,
	659,
	660,
	661,
	662,
	665,
	666,
	669,
	670,
	672,
	673,
	674,
	675,
	676,
	677,
	678,
	679,
	680,
	681,
	682,
	683,
	684,
	685,
	686,
	687,
	688,
	689,
	690,
	693,
	694,
	695,
	696,
	697,
	698,
	699,
	701,
	702,
	703,
	704,
	705,
	706,
	708,
	709,
	710,
	711,
	712,
	713,
	714,
	715,
	716,
	717,
	719,
	720,
	727,
	728,
	729,
	730,
	731,
	732,
	733,
	734,
	735,
	736,
	737,
	738,
	742,
	745,
	746,
	747,
	748,
	752,
	753,
	755,
	756,
	757,
	758,
	760,
	761,
	762,
	763,
	764,
	765,
	766,
	767,
	768,
	769,
	770,
	771,
	772,
	773,
	774,
	775,
	776,
	778,
	779,
	780,
	781,
	782,
	783,
	784,
	786,
	787,
	788,
	789,
	790,
	791,
	792,
	793,
	794,
	795,
	796,
	797,
	798,
	799,
	800,
	802,
	806,
	812,
	813,
	816,
	817,
	819,
	826,
	827,
	828,
	831,
	834,
	836,
	837,
	838,
	839,
	840,
	841,
	842,
	843,
	844,
	846,
	847,
	848,
	852,
	853,
	854,
	857,
	858,
	859,
	861,
	862,
	863,
	864,
	865,
	866,
	869,
	870,
	871,
	872,
	873,
	874,
	876,
	877,
	878,
	879,
	881,
	883,
	884,
	885,
	887,
	888.}

\normalsize

\noindent \textbf{11 and 12 crossing knots with tunnel number three:}

\noindent \tiny{12a$ X $ with $ X =$ 554,
	750.}

\noindent \tiny{12n$ X $ with $ X =$ 553,
	554,
	555,
	556,
	642.}

\normalsize

\noindent \textbf{11 and 12 crossing knots with tunnel number $ \in \brackets{1,2} $:}

\noindent \tiny{11n$ X $ with $ X = $ 45,
	65,
	66,
	86,
	92,
	94,
	95,
	98,
	99,
	112,
	113,
	115,
	117,
	118,
	119,
	121,
	123,
	125,
	127,
	136,
	142,
	144,
	146,
	149,
	150,
	152,
	153,
	155,
	156,
	158,
	161,
	163,
	168,
	169,
	170,
	171,
	173,
	178,
	179,
	180,
	181,
	182.}

\noindent \tiny{12n$ X $ with $ X = $ 16,
	52,
	53,
	88,
	156,
	157,
	158,
	204,
	207,
	211,
	228,
	230,
	245,
	246,
	247,
	258,
	260,
	279,
	280,
	282,
	283,
	284,
	285,
	287,
	299,
	311,
	314,
	321,
	328,
	329,
	330,
	354,
	373,
	375,
	377,
	383,
	391,
	392,
	395,
	411,
	416,
	417,
	425,
	426,
	432,
	438,
	439,
	441,
	443,
	445,
	449,
	450,
	451,
	452,
	453,
	458,
	459,
	481,
	482,
	486,
	489,
	490,
	499,
	511,
	519,
	521,
	524,
	525,
	527,
	535,
	537,
	538,
	542,
	543,
	544,
	548,
	551,
	580,
	585,
	587,
	595,
	596,
	603,
	610,
	612,
	613,
	615,
	625,
	628,
	629,
	632,
	638,
	639,
	653,
	663,
	664,
	667,
	668,
	671,
	691,
	692,
	700,
	707,
	718,
	739,
	740,
	741,
	743,
	744,
	750,
	751,
	754,
	759,
	777,
	785,
	801,
	803,
	804,
	805,
	807,
	808,
	809,
	810,
	811,
	814,
	815,
	818,
	820,
	821,
	822,
	823,
	824,
	825,
	829,
	830,
	832,
	833,
	835,
	845,
	849,
	850,
	855,
	856,
	860,
	867,
	868,
	875,
	880,
	882,
	886.}

\normalsize


\bibliography{References}
\bibliographystyle{aomplain}

\end{document}